\documentclass[a4paper,reqno]{amsart}

\usepackage[T1]{fontenc}
\usepackage[utf8]{inputenc}
\usepackage[unicode]{hyperref}

\usepackage{amsmath,amssymb,amsthm,mathtools}
\usepackage{cite}
\usepackage{tikz, tikz-cd, stmaryrd}

\usepackage{enumitem}
\setlist[enumerate]{label = \textup{\textup{(\alph*)}},ref = \textup{(\alph*}),leftmargin=\parindent+2ex,align=left,labelwidth=3ex,itemsep=\topsep}
\setlist[itemize]{itemsep=\topsep}

\theoremstyle{plain}
\newtheorem{thm}{Theorem}[section]
\newtheorem{prop}[thm]{Proposition}
\newtheorem{lem}[thm]{Lemma}

\theoremstyle{definition}
\newtheorem{dfn}[thm]{Definition}
\newtheorem{ex}[thm]{Example}
\newtheorem{rem}[thm]{Remark}
\newtheorem{nota}[thm]{Notation}
\numberwithin{equation}{section}
\def\al{\alpha}
\def\be{\beta}
\def\ga{\gamma}
\def\de{\delta}
\def\io{\iota}
\def\ep{\epsilon}
\def\la{\lambda}

\def\th{\theta}
\def\ze{\zeta}

\def\si{\sigma}

\def\De{\Delta}

\def\La{\Lambda}
\def\Om{\Omega}

\def\Th{\theta}

\def\Z{{\mathbin{\mathbb Z}}}
\def\Q{{\mathbin{\mathbb Q}}}

\def\C{{\mathbin{\mathbb C}}}

\def\A{{\mathbin{\mathcal A}}}
\def\M{{\mathbin{\mathcal M}}}
\def\cD{{\mathbin{\mathcal D}}}

\def\E{{\mathbin{\mathcal E}}}
\def\X{X}

\def\ra{\rightarrow}

\def\longra{\longrightarrow}

\DeclareMathOperator\colim{colim}

\SetSymbolFont{stmry}{bold}{U}{stmry}{m}{n}
\DeclareFontFamily{U}{rsfs}{} 
\DeclareFontShape{U}{rsfs}{n}{it}{<->
rsfs10}{} \DeclareSymbolFont{mscr}{U}{rsfs}{n}{it}
\DeclareSymbolFontAlphabet{\scr}{mscr}
\def\mathscr{\scr}

\def\O{\mathrm{O}}\def\U{\mathrm{U}}\def\SU{\mathrm{SU}}

\DeclareMathOperator\id{id}

\def\e#1\e{\begin{equation}#1\end{equation}}
\def\ea#1\ea{\begin{align}#1\end{align}}

\def\cL{\mathcal{L}}
\def\CP{\mathbb{CP}}
\def\P{\mathbb{P}}

\def\pt{\mathrm{pt}}

\def\op{\oplus}
\def\ot{\otimes}
\def\iy{\infty}
\def\longra{\longrightarrow}
\def\t{\times}
\def\bt{\boxtimes}
\def\an#1{\langle #1 \rangle}

\def\cS{\mathcal{S}}

\def\top{\mathrm{top}}
\DeclareMathOperator\rk{rk}

\def\vac{\Om}

\title[Vertex F-algebras and complex oriented homology]{Vertex F-algebra structures on the\\complex oriented homology of H-spaces}
\author{Jacob Gross and Markus Upmeier}

\date{\today}


\begin{document}

\begin{abstract}
We give a topological construction of graded vertex $F$-algebras that generalizes Joyce's vertex algebra to complex-oriented homology. Given an H-space $\X$ with a $B\U(1)$-action, a choice of  K-theory class, and a complex oriented homology theory $E,$ we build a graded vertex $F$-algebra structure on $E_*(\X)$ where $F$ is the formal group law associated with $E.$
\end{abstract}

\maketitle

\setcounter{tocdepth}{1}
\tableofcontents

\section{Introduction and Results}
\label{s1}


The algebraic topology of moduli stacks, arising for example in algebraic geometry and gauge theory, is of fundamental importance for the study of invariants. Let $\A$ be an additive $\C$-linear dg-category, whose $\tau$-stable objects we wish to classify, for $\tau$ a stability condition. The category $\A$ has an associated moduli stack $\M_\A$ by~\cite{ToVa}. In \cite{Joy}, Joyce constructs a graded vertex algebra on the ordinary homology $H_*(\M_\A).$ Vertex algebras are algebraic structures with origins in conformal field theory which can be regarded as \emph{singular} commutative rings whose operation ${Y\colon V\ot V\to V(\!(z)\!)},$ the \emph{state-to-field correspondence}, takes values in Laurent polynomials. This profound algebraic structure is used to describe wall-crossing formulas relating the virtual fundamental classes $[\M_\A]_\tau^\mathrm{virt}, [\M_\A]_{\tau'}^\mathrm{virt} \in H_*(\M_\A)$ for different stability conditions. These are powerful tools for computing invariants.

Motivated by physics, many authors currently investigate \emph{refined invariants} such as $K$-theoretic Donaldson--Thomas invariants \cite{GNY,GKW,KS,Th}. Here the virtual classes should viewed in K-homology $K_*(\M_\A).$ As a first step towards extending wall-crossing formulas to refined invariants, we here extend Joyce's construction to any generalized (complex oriented) homology theory $E_*$ with associated formal group law $F(z,w).$ Our main result constructs a vertex $F$-algebra structure on $E_*(\M_\A)$ in the sense of Li~\cite{Li}.

In addition, our construction of vertex $F$-algebra works in greater generality, namely for any topological H-space (\emph{i.e.}~abelian group up to homotopy) with an action of $B\U(1).$ 

Let $E^*$ be a complex oriented generalized cohomology theory with associated formal group law\/ $F(z,w)$ over its coefficient ring\/ $R_*,$ see \textup{\S\ref{s3}}. As a preliminary result, we present a Laurent-polynomial version of the Conner-Floyd Chern classes (see Definition~\textup{\ref{s3dfn3}}) with values in $E^*.$



\begin{thm}
\label{s1thm1} For every class\/ $\Th\in K^0(X)$ in the topological $K$-theory of topological space\/ $X$ there is an\/ $R$-linear transformation
\ea
\label{s1eq2}
 (-)\cap C_z^E(\th)\colon E_*(X)&\longra E_*(X)\llbracket z\rrbracket[z^{-1}]
 & a\longmapsto a\cap C_z^E(\th),
\ea
of degree\/ $-2r$ if\/ $\th$ has constant rank\/ $r\in\Z,$ with the following properties:
\begin{enumerate}
\item\textup(Naturality.\textup)~For continuous\/ $f\colon X'\to X,$ $\th\in K^0(X),$ and\/ $a'\in E_*(X')$
\e
\label{s1eq3}
 f_*\bigl(a'\cap C_z^E(f^*(\th))\bigr)=f_*(a')\cap C_z^E(\th).
\e
\item\textup(Direct sums.\textup)~For\/ $\ze,$ $\th\in K^0(X)$ and\/ $a\in E_*(X)$ we have
\e
\label{s1eq4}
 a\cap C_z^E(\ze+\th)=\bigl[a\cap C_z^E(\ze)\bigr]\cap C_z^E(\th).
\e
\item\textup(Normalization.\textup)~ 
For a complex line bundle\/ ${L\to X}$ and\/ $a\in E_*(X)$ we have
\e
\label{s1eq5}
 a\cap C_z^E(L)=a\cap F(z,c_1^E(L)).
\e
More generally, for any\/ ${\th\in K^0(X)}$ we have
\e 
\label{s1eq6}
a\cap C^E_z(L\ot\th) = i_{z, c^E_1(L)} \bigl(a\cap C^E_{F(z,c^E_1(L))}(\th)\bigr).
\e
\end{enumerate}
\end{thm} 

Here, as usual, the variable $z$ has degree $-2.$ We prove Theorem~\ref{s1thm1} in \S\ref{s4}. The notations used in \eqref{s1eq5} and \eqref{s1eq6} will be explained in Notations~\ref{s3nota1} \& \ref{s4nota1} below.
\medskip

For our main result, let $X$ be an H-space with an operation ${\Phi\colon X\t X\to X}$ that is associative, commutative, and has a unit $e\in X$ up to homotopy. Recall that the classifying space $B\U(1)$ for complex line bundles is an H-space with the tensor product $\mu_{B\U(1)}$ and trivial bundle $e_{B\U(1)}.$ Assume there is an action ${\Psi}$ of $B\U(1)$ on $X$ up to homotopy, meaning $\Psi\circ(\id_{B\U(1)}\t\Psi)\simeq\Psi\circ(\mu_{B\U(1)}\t\id_X)$ and $\Psi(e_{B\U(1)},-)\simeq{\id_X}.$ Suppose\/ $\Psi(e,-)\simeq e$ is an h-fixed point and ${\Phi\circ(\Psi\t\Psi)\circ\de}\simeq{\Psi\circ(\Phi\t\id_{B\U(1)})},$ where $\de(x_1,x_2,g)=(x_1,g,x_2,g).$ The set of connected components $\pi_0(X)$ is a monoid with unit $\vac=[e]$ and operation $\al+\be=\Phi_*(\al\boxtimes\be)$ and we partition $X=\coprod_{\al\in\pi_0(X)} X_\al.$ Write ${\Phi_{\al,\be}\colon X_\al\t X_\be\to X_{\al+\be}},$ ${\Psi_\al\colon B\U(1)\t X_\al}\to X_\al$ for the restrictions. Let $\th_{\al,\be}\in K^0(X_\al\t X_\be)$ for all $\al,\be.$

\begin{thm}
\label{s1thm2}
Given\/ $(X,\Phi,e,\Psi)$ as above, suppose the following identities hold  for all\/ $\al,\be, \ga \in \pi_0(X)$:
\begin{align}
\label{s1eq7}
(\Phi_{\al,\be}\t\id_{X_\ga})^*(\th_{\al+\be,\ga})&=\pi^*_{\al,\ga}(\th_{\al,\ga}) + \pi^*_{\be,\ga}(\th_{\be,\ga}),\\
\label{s1eq8}
(\id_{X_\al}\t\Phi_{\be,\ga})^*(\th_{\al,\be+\ga})&=\pi^*_{\al,\be}(\th_{\al,\be}) + \pi^*_{\al,\ga}(\th_{\al,\ga}),\\
\label{s1eq9}
(\Psi_\al\t\id_{X_\be})^*(\th_{\al,\be})&=\pi_{B\U(1)}^*(\mathcal{L})^{\;}\,\ot \pi^*_{\al,\be}(\th_{\al,\be}),\\
\label{s1eq10}
(\id_{X_\al}\t\Psi_\be)^*(\th_{\al,\be})&=\pi_{B\U(1)}^*(\mathcal{L})^{\;\mathclap{\vee}}\,\ot\pi_{\al,\be}^*(\th_{\al,\be}),\\
\label{s1eq11}
\th|_{X_\al\t\{\vac\}}&=0,\quad\th|_{\{\vac\}\t X_\be}=0,\\
\label{s1eq12}
\si^*(\th_{\be,\al})&=(\th_{\al,\be})^\vee.
\end{align}
Here\/ $\si$ swaps the factors of\/ ${X_\al\t X_\be}$ and\/ $\mathcal{L}\to B\U(1)$ is the universal line bundle with dual\/ $\mathcal{L}^\vee.$ With the\/ $F$-shift operator\/ $\cD(z)$ of~\eqref{s3eq3} below, the graded\/ $R_*$-module
\e
\label{s1eq13}
 V_*=\bigoplus\nolimits_{\al\in\pi_0(X)} E_{*-\rk\th_{\al,\al}}(X_\al)
\e
is a graded nonlocal vertex\/ $F$-algebra\/ $(V_*,\cD,\vac,Y)$ with state-to-field correspondence
\begin{equation}
\label{s1eq14}
    Y(a,z)b = (\Phi_{\al,\be})_* \bigl(\cD_\al(z)\bt\id_{E_*(X_\be)}\bigr)\bigl[(a\bt b)\cap C^E_z(\th_{\al,\be})\bigr].
\end{equation}
Similarly, the graded\/ $R_*$-module
\e
	\overline{V}_*=\bigoplus\nolimits_{\al\in\pi_0(X)} E_{*-2\rk\th_{\al,\al}}(X_\al)
\e
becomes a graded vertex\/ $F$-algebra $(\overline{V}_*,\cD,\vac,\overline Y),$ where
\begin{equation}
\label{s1eq15}
 \overline{Y}(a,z)b = (\Phi_{\al,\be})_* \bigl(\cD_\al(z)\bt\id_{E_*(X_\be)}\bigr)\bigl[(a\bt b) \cap \overline{C}^E_z(\th_{\al,\be})\bigr]
\end{equation} 
uses the operation of degree\/ $-4\rk\th_{\al,\be}$ defined by
\[
c\cap \overline{C}^E_z(\th_{\al,\be})=\bigl[c\cap C^E_z(\th_{\al,\be})\bigr]\cap C^E_{\iota(z)}(\si^*(\th_{\be,\al})),\quad c\in E_*(X_\al\t X_\be).
\]
\end{thm}

\noindent
Here\/ $\io(z)$ is the inverse for\/ $F$ \textup(see~\textup{\S\ref{s2}}\textup). The proof of Theorem~\ref{s1thm2} is given in \S\ref{s5}.\medskip

As a special case, our result applies to the topological realization $X=\M_\A^\top$ of a moduli stack. Taking direct sums in the additive category defines $\Phi$ making $\M_\A$ into an H-space. Moreover, scaling morphism by $\U(1)$ defines an operation $\Psi$ of the quotient stack ${[*\sslash\U(1)]},$ endowing $\M_\A^\top$ with the required action of $B\U(1)={[*\sslash\U(1)]}^\top.$ As shown in Proposition~\ref{s3prop1} below, this action yields an $F$-shift operator $\cD(z).$ The $K$-theory classes $\th_{\al,\be}$ are given by the Ext-complexes in the dg-category $\A,$ which satisfy \eqref{s1eq7}--\eqref{s1eq12}. In geometric examples, one may wish to incorporate signs $\ep_{\al,\be}$ into \eqref{s1eq15}. These are related to orientations, see \cite[\S8.3]{Joy}. The orientation problems were solved in the series \cite{JTU,JoUp1,JoUp2}. For simplicity, we ignore this additional data here and set up a symmetrized construction without signs.


\section{Formal Groups Laws and Vertex F-algebras}
\label{s2}

\begin{nota}
\label{s1nota1}
\hangindent\leftmargini
$\bullet$\hskip\labelsep $R_*=R^{-*}$ a graded commutative ring with unit
\begin{itemize}
\item
$V_*$ a graded module over $R_*$
\item
$z,w$ variables of degree $-2$
\item
$F(z,w)$ a graded formal group law over $R_*$
\item
$V\llbracket z\rrbracket$ the formal power series $\sum_{i=0}^\iy a_i z^i;$ a ring when $V=R$
\item
$V(\!(z)\!)$ the group of Laurent series $\sum_{i=-\iy}^{+\iy} a_i z^i$ with its partially defined product. The fact that $V(\!(z)\!)$ is \emph{not} a ring frequently causes confusion.
\item
The meromorphic series $V\llbracket z\rrbracket[z^{-1}];$ a ring when $V=R.$
\item
$i_{z,w}\colon V\llbracket z,w\rrbracket[z^{-1},w^{-1},F(z,w)^{-1}]\to V(\!(z,w)\!)$ expands $F(z,w)^{-N},$ see Notation~\ref{s2nota1}. We have $i_{z,w}(V\llbracket z,w\rrbracket[F(z,w)^{-1}])\subset V(\!(z)\!)\llbracket w\rrbracket.$
\item
$(-1)^a$ means $(-1)^{{\rm degree}(a)}$
\end{itemize}
\end{nota}

\begin{dfn}
\label{s2dfn1}
A \emph{graded formal group law} over $R_*$ is a formal power series $F(z,w)=\sum_{i,j\geqslant0}F_{ij}z^iw^j \in R\llbracket z,w\rrbracket$ with $F_{ij}\in R_{2i+2j-2}$ satisfying
\ea
\label{s2eq1}
F(z,w) &= F(w,z),
&F(z,0) &= z,
&F(F(z,w),v) &= F(z, F(w,v)).
\ea

There exists a unique power series $\io\in R\llbracket z\rrbracket$ with $F(z,\io(z)) = 0,$ the \emph{inverse}. Note that $\io(\io(z)) = z$ and $\io(F(\io(z),w)) = F(z, \io(w)).$
\end{dfn}


\begin{ex}
\label{s2ex1}
\hangindent\leftmargini
\textup{(i)}
The \emph{additive formal group law} $\mathbb{G}_\mathrm{a}$ over $\Z$ (in degree zero) is defined by $F(z,w) = z + w,$ and the inverse is $\io(z) = -z.$
\begin{enumerate}
\setcounter{enumi}{1}
 \item[(ii)] The \emph{multiplicative formal group law} $\mathbb{G}_\mathrm{m}$ over $\Z$ is defined by $F(z,w) = z + w + zw$ and has $\io(z) = (1+z)^{-1}-1 = -z + z^2 - z^3 + \cdots.$
 \item[(iii)] There is a \emph{universal formal group law} $\mathbb{G}_\mathrm{u}$ over the \emph{Lazard ring} $R_L$ generated by variables $F_{ij}$ subject to the relations contained in \eqref{s2eq1}.
\end{enumerate}
\end{ex}

\begin{nota}
\label{s2nota1}
It follows from \eqref{s2eq1} that for a general formal group law
\begin{align*}
 F(z,w) &= z + w + O(zw),
&\io(z)&= -z + O(z^2).
\end{align*}
Write $F(z,w)=z(1+w/z+wG(z,w))$ and expand using the binomial theorem
\e
\label{s2eq2}
 i_{z,w}F(z,w)^n=\sum_{k=0}^\iy\binom{n}{k} z^{n-k}w^k(1+zG(z,w))^k\in R\llbracket w\rrbracket(\!(z)\!),\quad n\in\Z.
\e
As the $k$-th summand has $w$-degree $\geqslant k,$ this converges as a formal power series. Define $i_{w,z}F(z,w)^n\in R\llbracket z\rrbracket(\!(w)\!)$ by expanding $F(z,w)=w(1+z/w+zG(z,w))$ similarly. We extend $i_{z,w}$ and $i_{w,z}$ to $V\llbracket z,w\rrbracket[z^{-1},w^{-1}][F(z,w)^{-1}]$ by linearity.

Note that $i_{z,w}F(z,w)^n\cdot F(z,w)^n=1$ and $i_{w,z}F(w,z)^n\cdot F(z,w)^n=1.$ For every $P(z,w)=\!\sum_{n\geqslant-N}\!a_n(z,w)F(z,w)^n\in V\llbracket z,w\rrbracket[z^{-1},w^{-1}][F(z,w)^{-1}]$ we thus have
\e
\label{s2eq3}
 F(z,w)^N\bigl(i_{z,w}P(z,w)-i_{w,z}P(z,w)\bigr)=0.
\e
\end{nota}


\begin{dfn}
\label{s2dfn3}
Let $V_*$ be a graded $R_*$-module and $F$ a graded formal group law over $R_*.$ An \emph{$F$-shift operator} is a graded $R$-linear map $\cD(z)\colon V \ra V\llbracket z\rrbracket$ with
\begin{align}
 \label{s2eq4}
 \cD(0)&=\id_V,
 &\cD(z) \circ \cD(w) &= \cD(F(z,w)).
\end{align}
\end{dfn}


\begin{ex}
\label{s2ex2}
Let $R=\Q,$ ${V = \Q[w]}.$ Then $\cD(z)(f(w)) = e^{z \frac{d}{dw}}f(w)$
defines a $\mathbb{G}_\mathrm{a}$-shift operator. The relation $\cD(z)(f(w)) = f(z+w)$ motivates the terminology.
\end{ex}

We now define vertex $F$-algebras. For $F=\mathbb{G}_\mathrm{a}$ we recover ordinary vertex algebras, see Frenkel--Ben-Zvi \cite{FrBZ}, Frenkel--Lepowsky--Meurman \cite{FLM}, and Kac \cite{Kac}. 


\begin{dfn}
\label{s2dfn5}
Let $F(z,w)$ be a graded formal group law over $R_*.$ A \emph{graded nonlocal vertex $F$-algebra} is a graded $R_*$-module $V_*,$ a {\it vacuum vector} $\vac\in V_0,$ an $F$-shift operator $\mathcal{D}(z),$ and a graded $R$-linear \emph{state-to-field correspondence}
\ea
\label{s2eq5}
 V\ot_R V &\longra V\llbracket z\rrbracket[z^{-1}], &a\ot b\longmapsto Y(a,z)b,
\ea
satisfying the following axioms:
\begin{enumerate}
	\item \emph{Vacuum and creation:} $Y(a,z)\vac$ is holomorphic for all $a\in V$ and
	\ea
	\label{s2eq6}
	 \left.Y(a,z)\vac\right|_{z=0} &= a,\\
	\label{s2eq7}
          Y(\vac,z) &= \id_V.
	\ea
	\item {\it $F$-translation covariance:} for all $a \in V$ we have
	\ea
	\label{s2eq8}
	 Y(\cD(w)(a),z) &= i_{z,w} Y(a, F(z,w)),\\
	\label{s2eq9}
	 \cD(z)\vac &= \vac.
	\ea
	\item \emph{Weak $F$-associativity:} for all $a,b,c \in V$ there exists $N \geqslant 0$ with
        \e
        \label{s2eq10}
         F(z,w)^N Y(Y(a,z)b,w)c = F(z,w)^N i_{z,w} Y(a, F(z,w)) Y(b,w) c.
        \e
\end{enumerate}
A graded nonlocal vertex $F$-algebra is a {\it graded vertex $F$-algebra} if, in addition,
\ea
 \label{s2eq11}
  Y(a,z)b &= (-1)^{ab} \cD(z) \circ Y(b, \io(z))a,
  &&\text{for all $a,b \in V$.}
\ea
\end{dfn}

\begin{rem}
It is a consequence of \eqref{s2eq6}--\eqref{s2eq11} that for all $a,b,c\in V$ there exists $N\geqslant0$ with
\e
\label{s2eq12}
 (z-w)^N Y(a,z)Y(b,w)c=(-1)^{ab}(z-w)^NY(b,w)Y(a,z)c.
\e
So our definitions agree with those given by Li~\cite{Li} in the ungraded case.
\end{rem}

\section{Complex Oriented Cohomology and Chern Classes}
\label{s3}

Let $E^*$ be a generalized cohomology theory, see for example Rudyak \cite[Ch.~II, \S3]{Rud}. Thus, for every pair $A\subset X$ of topological spaces there is defined a graded abelian group $E^*(X,A).$ Continuous maps $f\colon(X,A)\to(X',A')$ induce homomorphisms $f^*\colon E^*(X',A')\to E^*(X,A)$ that depend only on the homotopy class of $f.$ For a pointed space $x_0\in X$ write $\widetilde{E}^*(X)=E^*(X,\{x_0\})$ for \emph{reduced cohomology}. The \emph{smash product} of $(X,x_0)$ and $(Y,y_0)$ is the quotient $X\wedge Y=(X\t Y)/(X\vee Y)$ with one-point union $X\vee Y=(X\t\{y_0\})\cup(\{x_0\}\t Y)$ collapsed to become the new base-point. As part of the structure, $E^*$ comes equipped with natural \emph{suspension isomorphisms} $\si_X\colon \widetilde{E}^*(X)\to\widetilde{E}^{*+1}(X\wedge\cS^1).$

Suppose $E^*$ is a multiplicative generalized cohomology theory. Then there is a bilinear \emph{cross product} $\boxtimes\colon E^*(X,A)\ot E^*(Y,B)\to E^*(X\t Y,X\t B\cup A\t Y)$ and \emph{units} $1_X\in E^0(X),$ both natural. Pulling the cross product back along the diagonal makes $E^*(X)$ a graded commutative unital $R^*$-algebra for the cup product `$\cup$' over the \emph{coefficient ring} $R^*=E^*(\pt).$ Dually, there is a homological cross product that in particular makes $E_*(X)$ a graded module over $R_*=R^{-*}.$ There is a \emph{cap product}
\[
 E_a(X)\ot_R E^b(X)\longra E_{a-b}(X),\qquad a\ot\varphi\mapsto a\cap\varphi
\]
which is $R$-linear, unital $a\cap 1=a,$ and natural $f_*(a\cap f^*(\varphi'))=f_*(a)\cap\varphi',$ where $f\colon X\to X'$ and $\varphi'\in E^b(X').$ See Rudyak~\cite{Rud} for further properties.

\begin{dfn}
\label{s3dfn1}
The suspension isomorphism shows that $\widetilde{E}^*(\CP^1)\cong\widetilde{E}^*(\cS^2)\cong R^{*-2}$ is a free $R$-module on a single generator. A multiplicative cohomology theory $E^*$ is \emph{complex orientable} if $i^*\colon\widetilde{E}^*(\CP^\infty)\to\widetilde{E}^*(\CP^1)$ is surjective, where $\CP^\iy\!\cong\operatorname{colim}_m\CP^m$ and $i\colon\CP^1\hookrightarrow\CP^\iy.$
A \emph{complex orientation} is a choice of $\xi_E\in \widetilde{E}^2(\CP^\infty)$ such that $i^*(\xi_E)$ generates the $R$-module $\widetilde{E}^*(\CP^1).$
\end{dfn}

The presence of the permanent cycle $\xi_E\vert_{\CP^m}$ implies that the Atiyah--Hirzebruch spectral sequence $H^p(\CP^m;E^q(\pt))\Longrightarrow E^{p+q}(\CP^m)$ collapses, see Adams~\cite[p.~42]{Adams}. Hence we have canonical isomorphisms
\begin{align*}
E^*(\CP^m)&\cong R[\xi_E]/(\xi_E^{m+1}),
&E^*(\CP^\iy)&\cong\lim E^*(\CP^m)\cong R\llbracket \xi_E\rrbracket.
\end{align*}
More generally, let $P\to X$ be a bundle of projective spaces $\CP^m$ and suppose that $w\in E^*(P)$ restricts on every fiber $P_x$ to generators $1_{P_x}, w|_{P_x}, \ldots, w^m|_{P_x}$ of the $R$-module $E^*(P_x).$ Then Dold's theorem implies that $E^*(P)$ is a free $E^*(X)$-module on $1_P,w,\ldots, w^m,$ see \cite[(7.4)]{CoFl}. In particular,
\ea
\label{s3eq1}
 E^*(X\t\CP^\iy)&\cong E^*(X)\llbracket \xi_E\rrbracket,
&E^*({\CP^\iy\!\t\CP^\iy})&\cong R\llbracket \pi_1^*(\xi_E),\pi_2^*(\xi_E)\rrbracket.
\ea

\begin{dfn}
\label{s3dfn2}
 Let $\xi_E$ be a complex orientation of $E^*.$ Write $\cL\to \CP^\iy$ for the universal complex line bundle with $\cL|_L=L.$ Recall that $\CP^\iy=B\U(1)$ is an H-space with operation a classifying map ${\mu_{\CP^\iy}\colon{\CP^\iy\!\t\CP^\iy}\to\CP^\iy}$ of the tensor product ${\pi_1^*(\cL)\ot\pi_2^*(\cL)}$ and unit $t_0$ the trivial line bundle. The associated \emph{formal group law} $F=\sum_{i,j\geqslant0}F_{ij}z^iw^j$ is defined by the expansion
 \e
 \label{s3eq2}
 \mu_{\CP^\iy}^*(\xi_E)=\sum\nolimits_{i,j\geqslant0} F_{ij}\, \xi_E^i\bt\xi_E^j,\qquad F_{ij}\in R^{2-2i-2j}=R_{2i+2j-2}.
 \e
\end{dfn}

As in \cite[p.~42]{Adams} the homology $E_*(\CP^\iy)$ is the free $R$-module on the dual generators $t_n,$ $n\geqslant0,$ of degree $2n$ characterized by $\an{t_n,\xi_E^m}=\de_n^m.$ 

\begin{prop}
\label{s3prop1}
Let\/ $(E^*,\xi_E)$ be a complex oriented cohomology theory and associated formal group law\/ $F(z,w).$ Suppose\/ $\Psi\colon B\U(1)\t X\to X$ satisfies the axioms for a group action of the H-space\/ $B\U(1)$ on\/ $X$ up to homotopy. Then
\ea
\label{s3eq3}
 \cD(z)(a)&=\sum\nolimits_{k\geqslant 0} \Psi_*(t_k\bt a)\, z^k,
 &&a\in E_*(X),
\ea
defines an\/ $F$-shift operator on\/ $E_*(X).$
\end{prop}

\begin{proof}
Since $\Psi(t_0,x)=x$ is neutral, $\cD(0)=\id_{E_*(X)}.$ Define coefficients $F_{ij}^n$ by $F(z,w)^n=\sum_{i,j\geqslant0} F_{ij}^n\,z^iw^j.$ Then $(\mu_{\CP^\iy})_*(t_i\bt t_j)=\sum_{n\geqslant 0} F_{ij}^n\,t_n,$ and so
\begin{align*}
 \cD(z)\circ\cD(w)(a)&=\sum\nolimits_{\,\,\,\quad\mathclap{i,j\geqslant0}}\;\quad\Psi_*\bigl(t_i\bt \Psi_*(t_j\bt a)\bigr)z^iw^j\\
 &=\sum\nolimits_{\,\,\,\quad\mathclap{i,j\geqslant0}}\;\quad\Psi_*\bigl((\mu_{\CP^\iy})_*(t_i\bt t_j)\bt a\bigr)z^iw^j\\
 &=\sum\nolimits_{\,\,\,\quad\mathclap{i,j,n\geqslant0}}\;\quad\Psi_*(t_n\bt a) F_{ij}^n\,z^iw^j=\cD(F(z,w)).\qedhere
\end{align*}
\end{proof}

\begin{dfn}
\label{s3dfn3}
 Let $V\to X$ be a complex vector bundle of rank $n$ with zero section $0_X.$ The  bundle of projective spaces $\P(V)=(V\setminus0_X)/\C^*$ carries a tautological line bundle $\cL_V \to \P(V)$ with $\cL_V\vert_L=L.$ Its classifying map $f_{\cL_V}\colon \P(V)\to\CP^\iy$ is unique up to homotopy. Define $w=f_{\cL_V}^*(\xi_E)$ using the complex orientation. By the above, $E^*(\P(V))$ is a free $E^*(X)$-module with basis $1_{\P(V)},w,\ldots,w^{n-1}.$ The \emph{Conner--Floyd Chern classes} are defined by expanding $w^n$ in this basis:
\ea
\label{s3eq4}
c_0^E(V)&=1,
&0&=\sum\nolimits_{i=0}^n (-1)^ic_i^E(V)\cdot w^{n-i},
&c_i^E(V)&=0\enskip(\forall i>n)
\ea
Naturality under pullback is obvious. There is a Whitney sum formula \cite[p.~47]{CoFl}
\e
\label{s3eq5}
c_k^E(V\op W)=\sum\nolimits_{i=0}^k c_i^E(V)c_{k-i}^E(W).
\e
For complex line bundles $\cL_L\to \P(L)$ is isomorphic to $L\to X$ so $c_1^E(L)=f_L^*(\xi_E)$ for the classifying map $f_L$ of $L.$ In particular, 
\e
\label{s3eq6}
c_1^E(L_1\ot L_2)=F(c_1^E(L_1),c_1^E(L_2)).
\e
Moreover, $c_1^E(\underline{\C})=0$ as $\xi_E$ is reduced. Hence $c_i^E(\underline{\C}^N)=0$ for every trivial bundle.
\end{dfn}

\begin{ex}
\label{s3ex1}
 Ordinary cohomology ${E^*=H^*}$ has a complex orientation $\xi_H$ in $H^2(\CP^\iy)=\lim H^2(\CP^m)$ that is Poincar\'e dual to the fundamental class $[\CP^{m-1}]\in H_{m-2}(\CP^m)$ with orientation of $\CP^{m-1}$ fixed by the complex structure. We obtain the ordinary Chern classes, and $c_1(L_1\ot L_2)=c_1(L_1)+c_1(L_2)$ implies $F_H=\mathbb{G}_\mathrm{a}.$
\end{ex}

\begin{ex}
\label{s3ex2}
 Topological $K$-theory ${E^*=K^*}$ on compact spaces is the group completion of isomorphism classes of complex vector bundles. Write $\cL_m=\cL|_{\CP^m}$ for the tautological complex line bundle over $\CP^m,$ $\C$ for the trivial bundle, and $[\cL_m], 1\in K^0(\CP^m)$ for their classes in $K$-theory. The classes ${[\cL_m]-1}\in\widetilde{K}^0(\CP^m)$ are compatible under restriction and define a complex orientation $\xi_K\in\widetilde{K}^2(\CP^\iy)=\lim\widetilde{K}^0(\CP^m).$ Here $F_K=\mathbb{G}_\mathrm{m}$ is the multiplicative formal group law, as
 \[
  \mu^*({[\cL]-1})={[\mu^*(\cL)]-1}={[\pi_1^*(\cL)\ot\pi_2^*(\cL)]-1}=\mathbb{G}_\mathrm{m}\bigl({[\cL]-1},{[\cL]-1}\bigr).
 \]
 
 For a complex vector bundle $V\to X$ of rank $n$ one has $\pi^*(V)=\cL_V\op\cL_V^\bot$ over the projectivization $\pi\colon\P(V)\to X$ and $\cL_V^\bot.$ The formal power series $\La_t([V])=1+[V]t+[\La^2 V]t^2+\ldots\in K^0(X)\llbracket t\rrbracket$ has inverse $\La_{-t}([V]),$ so $\La_t({[V]-[W]})=\La_t([V])\La_{-t}([W]).$ As $[\cL_V^\bot]={\pi^*[V]-[\cL]}$ has rank $n-1,$ the $n$-th coefficient of $\La_t([\cL_V^\bot])=\La_t([\pi^*(V)])\La_{-t}(\cL)$ is $0=[\La^n(\cL_V^\bot)]=\sum_{p=0}^n (-1)^{n-p}[\La^p(V)]\cdot[\cL]^{n-p}.$ Putting $[\cL]=w+1$ and comparing to \eqref{s3eq4}, $c_i^K(V)=\sum\nolimits_{p=0}^i (-1)^{i+p}\binom{n-p}{n-i}[\La^p(V)].$
\end{ex}

\begin{ex}
\label{s3ex3}
 As in Quillen~\cite{Q}, complex cobordism $\Om^n_\U(X)$ for $X$ a smooth manifold is the set of smooth maps ${f\colon Z\to X}$ of codimension $\dim X-\dim Z=n$ with a complex structure the stable normal bundle, modulo cobordism. The complex orientation $\xi_\Om\in\Om^2_\U(\CP^\iy)=\lim \Om^2_\U(\CP^m)$ is given by $\CP^{m-1}\hookrightarrow\CP^m,$ and $\CP^{m-1}$ is the zero set of a section of $\cL_m^*.$ So for complex line bundles $c_1^{\Om_\U}(L)$ is represented by the zero set $s^{-1}(0)$ of a generic section $s\colon X\to L.$ The formal group law is the universal law $\mathbb{G}_\mathrm{u},$ see Adams~\cite[Part~I,~\S8]{Adams}.
\end{ex}

\begin{lem}
\label{s3lem1}
 Let\/ $V\to X$ be a complex vector bundle over a finite CW complex. Then each of the Conner--Floyd Chern classes\/ $c_i^E(V)$ is nilpotent.
\end{lem}

\begin{proof}
There is a finite open cover $X=\bigcup_{\la=1}^N U_\la$ with $U_\la$ contractible and $V|_{U_\la}$ trivial. From the long exact sequence of the pair $(X,U_\la)$ we see that we may lift $c_i^E(V)$ along $j_\la^*\colon E^{2i}(X,U_\la)\to E^{2i}(X)$ to a class $x_\la\in E^{2i}(X,U_\la).$ The diagram
\[
\hskip-5cm\begin{tikzcd}
 \prod_{\la=1}^N E^{2i}(X,U_\la)\arrow[d,"\prod_{\la=1}^N j_\la^*"]\arrow[r,"\cup"] & |[text width=37.59pt, align=center]|E^{2iN}\bigl(X,\bigcup_{\la=1}^N U_\la\bigr)=E^{2iN}(X,X)=\{0\}\arrow[d,"j^*"]\\
 \prod_{\la=1}^N E^{2i}(X)\arrow[r,"\cup"] & E^{2iN}(X)
\end{tikzcd}
\]
commutes by naturality of `$\cup$', so $
c_i^E(V)^N=\prod_{\la=1}^N j_\la^*(x_\la)=j^*\bigl(\prod_{\la=1}^N x_\la\bigr)=0.$
\end{proof}

\begin{nota}
\label{s3nota1}
When $X$ is a finite CW complex, it follows that we may substitute $w$ by $c_1^E(L)$ in the formal group law $F(z,w).$ To define the right hand side of \eqref{s1eq5} also for infinite CW complexes $X,$ let $\{X_i \mid i\in I\}$ be the direct system of finite subcomplexes $X_i\subset X$ ordered by inclusion.
The \emph{pro-group $E$-cohomology} is the inverse limit $\hat{E}^*(X)=\lim E^*(X_i).$ The family of all restrictions $F(z,c_1^E(L|_{X_i}))$ determines an element we write $F(z,c_1^E(L))\in\hat{E}^*(X)\llbracket z\rrbracket.$ As homology and direct limits commute, see~\cite[Prop.~7.53]{Swi}, we have $E_*(X)=\colim E_*(X_i)$ and therefore a well-defined cap product $E_*(X)\ot \hat{E}^*(X)\to E_*(X).$ This defines \eqref{s1eq5} in general.
\end{nota}

\section{\texorpdfstring{Proof of Theorem~\ref{s1thm1}}{Proof of Theorem 1.1}}
\label{s4}

\subsection*{Step 1: Vector bundles over finite CW complexes}

For a complex line bundle $L\to X$ over a finite CW complex $X$ define $C_z^E(L)=F(z,c_1^E(L)).$ For $V\to X$ a rank $n$ complex vector bundle we proceed by the splitting principle. As in Definition~\ref{s3dfn3} over the projectivization $p\colon\mathbb{P}(V)\to X$ we can split off a line bundle from $p^*(V)$ and $p^*\colon E^*(Y)\to E^*(X)$ is injective. Iterating, we find $q\colon Y\to X$ and line bundles $L_1, \ldots, L_n \to Y$ with $L_1\op\cdots\op L_n = q^*(V)$ and $q^*\colon E^*(Y)\to E^*(X)$ is injective. By \eqref{s3eq5}, the class $q^*(c_k^E(V))$ is the $k$-th elementary symmetric polynomial in the \emph{Chern roots} $c_1^E(L_1), \ldots, c_1^E(L_1).$ As the expression
\e
\label{s4eq1}
F(z,c_1^E(L_1))\cup\cdots\cup F(z,c_1^E(L_n)) = q^*(C_z^E(V))
\e
is a symmetric polynomial in the Chern roots, the fundamental theorem of symmetric polynomials implies it has a (unique) preimage $C_z^E(V)$ in $E^*(X)\llbracket z\rrbracket.$ The map \eqref{s1eq2} is obtained by combining the class $C_z^E(V)$ with the cap product
\[
\cap\colon E_*(X)\ot E^*(X)\llbracket z\rrbracket\to E_*(X)\llbracket z\rrbracket.
\]

\noindent
(a)\hskip\labelsep For naturality, let $f\colon X'\to X$ and use the pullback $Q\colon Y'=X'\t_X Y\to X'$ with its canonical map $F\colon Y'\to Y$ to split $V'=f^*(V)$ as $Q^*(V')\cong F^*q^*(V)\cong F^*(L_1)\op\cdots\op F^*(L_n).$ Naturality of the Conner--Floyd Chern classes implies that the pullback ${F^*q^*(C_z^E(V))}={Q^*f^*(C_z^E(V))}$ of \eqref{s4eq1} along $F$ is ${Q^*C_z^E(V')}.$ Thus,
\e
 \label{s4eq2}
 C_z^E(f^*(V))=f^*(C_z^E(V)).
\e

\noindent
(b)\hskip\labelsep Let $V,\, W\to X$ be vector bundles. Pick $q\colon Y\to X$ such that both $q^*(V)=L_1\op\cdots\op L_n$ and $q^*(W)=S_1\op\cdots\op S_m$ split into line bundles with $q^*$ injective. Then $q^*C_z^E(V)$ equals \eqref{s4eq1}, $q^*C_z^E(W)=F(z,c_1^E(S_1))\cup\cdots\cup F(z,c_1^E(S_m)),$ and
\[
q^*C_z^E(V\op W)=F(z,c_1^E(L_1))\cup\cdots\cup F(z,c_1^E(S_m))=q^*C_z^E(V)\cup q^*C_z^E(W).
\]
Hence
\e
 \label{s4eq3}
 C_z^E(V\op W)=C_z^E(V)\cup C_z^E(W).
\e
This proves that cap product with $C_z^E(V)$ satisfies Theorem~\ref{s1thm1}(a)\&(b). Part (c) holds by construction. For (d), in the case of line bundles the operation ${(-)\cap F(z,c_1^E(L))}={\sum_{i,j\geqslant0}F_{ij} z^i[(-)\cap c_1^E(L)^j]}$ has degree $-2,$ as $F_{ij}\in R_{2i+2j-2}.$ It then follows from \eqref{s4eq1} that in general $(-)\cap C_z^E(V)$ has degree $-2\rk(V).$
\vskip\topsep

\noindent
(e)\hskip\labelsep Let $V\to X$ be a vector bundle, $L\to X$ a complex line bundle, and suppose $q^*(V)$ splits as above. Then $q^*(L\ot V)=(q^*(L)\ot L_1)\op\cdots\op(q^*(L)\ot L_n)$ and so
\begin{align*}
 q^*C_z^E(L\ot V) &\;=\; F\bigl(z,c_1^E(q^*(L)\ot L_1)\bigr)\cup\cdots\cup F\bigl(z,c_1^E(q^*(L)\ot L_n)\bigr)\\
 &\;\overset{\mathclap{\eqref{s3eq6}}}{=}\;F\bigl(z,F(q^*c_1^E(L),c_1^E(L_1))\bigr)\cup\cdots\cup F\bigl(z,F(q^*c_1^E(L),c_1^E(L_n))\bigr)\\
 &\;\overset{\mathclap{\eqref{s2eq1}}}{=}\;F\bigl(F(z,q^*c_1^E(L)),c_1^E(L_1)\bigr)\cup\cdots\cup F\bigl(F(z,q^*c_1^E(L)),c_1^E(L_n)\bigr)\\
 &\;=\;q^*C_{F(z,c_1^E(L))}^E(V).
\end{align*}
Hence
\e
 \label{s4eq4}
 C_z^E(L\ot V)=C_{F(z,c_1^E(L))}^E(V).
\e

\subsection*{Step 2: Extension to K-theory of finite CW complexes}

So far, we have constructed a homomorphism $C_z^E\colon{(\operatorname{Vect}(X),\op)}\to{(E^*(X)\llbracket z\rrbracket,\cup)}$ on the monoid of complex vector bundles ${V\to X}$ up to isomorphism over a finite CW complex. We claim that every class ${C_z^E(V)}$ is invertible in the larger ring ${E^*(X)\llbracket z\rrbracket[z^{-1}]}.$ Indeed, there exists a vector bundle ${W\to X}$ with ${V\op W}\cong\underline{\C}^N$ trivial and therefore ${C_z^E(V)\cup C_z^E(W)}={C_z^E(\underline{\C}^N)}={F(z,c_1^E(\underline{\C}))^N}=z^N.$ As $X$ is a finite CW complex, its topological $K$-theory is the group completion of ${(\operatorname{Vect}(X),\op)}$ whose universal property allows us to uniquely extend the homomorphism to $C_z^E\colon K^0(X)\to{(E^*(X)\llbracket z\rrbracket[z^{-1}],\cup)}.$ It is easy to check that properties (a)--(d) continue to hold.

\begin{nota}
\label{s4nota1}
As $X$ is a finite CW complex, we may write ${\th=[V]-[\underline{\C}^\ell]}.$ Expand $C_z^E(V)=\sum_{n\geqslant0}^\iy C_n(V)z^n.$ Then
\e
\label{s4eq5}
 C_z^E(\th)=\sum\nolimits_{n\geqslant0} C_n(V)z^{n-\ell}.
\e
In Notation~\ref{s1nota1} we have defined ${i_{z,w}\bigl(F(z,w)^{-\ell}\sum_{n\geqslant -\ell}^\iy C_n(V)F(z,w)^n\bigr)}$ as a holomorphic series in $w$ which we can substitute by the nilpotent $c_1^E(L),$ see Lemma~\ref{s3lem1}. This defines ${i_{z,c_1^E(L)}C_{F(z,c_1^E(L))}^E(\th)}\in E^*(X)\llbracket z\rrbracket[z^{-1}]$ for finite $X.$ When $X$ is infinite, the classes for the restrictions of $\th$ to all finite subcomplexes $X_i\subset X$ define $i_{z,c_1^E(L)}C_{F(z,c_1^E(L))}^E(\th)\in \hat{E}(X)(\!(z)\!)$ in pro-group $E$-cohomology, see Notation~\ref{s3nota1}.
\end{nota}

We prove (e). As just seen, $C_z^E(L)=F(z,c_1^E(L))$ is invertible in $E^*(X)\llbracket z\rrbracket[z^{-1}].$ Therefore $i_{z,w}F(z,c_1^E(L))^n=F(z,c_1^E(L))^n$ for all $n\in\Z.$ Using Notation~\ref{s4nota1},
\begin{align*}
 C_z^E(L\ot\th)&\overset{\mathclap{\eqref{s4eq3}}}{\;=\;}C_z(L\ot V) C_z(L)^{-\ell}\\
 &\overset{\mathclap{\eqref{s4eq4}}}{\;=\;} C_{F(z,c_1^E(L))}(V) F(z,c_1^E(L))^{-\ell}\\
 &\;= \sum\nolimits_{n\geqslant0} C_n(V)F(z,c_1^E(L))^{n-\ell} = i_{z,c_1^E(L)}C_{F(z,c_1^E(L))}^E(\th).
\end{align*}

\subsection*{Step 3: Infinite complexes}

Let $\{X_i \mid i\in I\}$ be the direct system of finite subcomplexes of a CW complex $X$ ordered by inclusion. Write $\io(i)\colon X_i\subset X$ and $\io(i,j)\colon X_i\subset X_j$ for the inclusions. For $\th\in K^0(X),$ Step 2 yields maps
\begin{equation}
\label{s4eq6}
\begin{tikzcd}[column sep=huge]
 E_*(X_i)\arrow[r,"\cap\, C_z(\io(i)^*\th)"] & E_*(X_i)\llbracket z\rrbracket[z^{-1}]\arrow[r,"{\io(i)_*}"] & E_*(X)\llbracket z\rrbracket[z^{-1}].
\end{tikzcd}
\end{equation}
By naturality, $\io(i,j)_*(a)\cap C_z^E(\io(j)^*\th)=\io(i,j)_*(a\cap C_z^E(\io(i)^*\th))$ so the maps \eqref{s4eq6} determine a homomorphism $E_*(X)\cong\colim E_*(X_i)\to E_*(X)\llbracket z\rrbracket[z^{-1}]$ on the colimit, using that homology and direct limits commute, see~\cite[Prop.~7.53]{Swi}. Equivalently, the restrictions $C_z^E(\th|_{X_i})$ define a class $C_z^E(\th)\in\hat{E}^*(X)(\!(z)\!)$ in pro-group $E$-cohomology. Using the cap product $E_*(X)\ot \hat{E}^*(X)(\!(z)\!)\to E_*(X)(\!(z)\!)$ we can define $(-)\cap C_z^E(\th)\colon E_*(X)\to E_*(X)(\!(z)\!)$ which, a priori, has a larger codomain.

Finally, properties (a)--(e) pass to the limit.

\subsection*{Step 5: General topological spaces}

By the CW approximation theorem, there is a CW complex $X'$ with a weak homotopy equivalence $f\colon X'\to X.$ Then
\[
 a\cap C_z(\theta) = f_*( f_*^{-1}(a)\cap C_z(f^*\theta) )
\]
is well-defined, since this equation holds for a homotopy equivalence $f\colon X'\to X'$ by \eqref{s1eq3}. With this definition, the properties (a)--(e) carry over to $X.$\qed

\section{\texorpdfstring{Proof of Theorem~\ref{s1thm2}}{Proof of Theorem 1.2}}
\label{s5}

We verify Definition~\ref{s2dfn5}(a)--(c) for the graded module $V_*=\bigoplus E_{*-\rk\th_{\al,\al}}(X_\al),$ vacuum vector $\vac=e_*(1),$ $F$-shift operator \eqref{s3eq3}, and state-to-field correspondence \eqref{s1eq14}. Here, $e\colon\pt\to X_0$ is the H-space unit and $1\in E_0(\pt)=R^0.$

Writing $|a|_V=|a|+\rk\th_{\al,\al}$ for the shifted degree, we have
\[
|Y(a,z)b|_V=|Y(a,z,b)|+\rk\th_{\al+\be,\al+\be}=(|a|-\rk\th_{\al,\al})(|b|-\rk\th_{\be,\be})=|a|_V\cdot|b|_V,
\]
so that $Y$ preserves the grading of $V_*.$
\vskip\topsep

\noindent
(a)\hskip\labelsep Let $a\in E_*(X_\al),$ $b\in E_*(X_\be).$ As $e$ is a fixed point, $\Psi_*(t_k\bt\vac)=0$ for $k>0$ and $\Psi_*(t_0\bt\vac)=\vac.$ Hence $\cD(z)\vac=\vac.$ Let ${\varphi=(e,\id_{X_\be})\colon X_\be\to X_\vac\t X_\be}.$ Then
\begin{align*}
(\vac\bt b)\cap C_z^E(\th_{\vac,\be})&=\varphi_*(b)\cap C_z^E(\th_{\vac,\be})\\
&=\varphi_*\bigl( b \cap \varphi^*C_z^E(\th_{\vac,\be})\bigr)
\overset{\smash{\eqref{s1eq11}}}{=}\varphi_*(b\cap 1)=\vac\bt b,
\end{align*}
and so $Y(\vac,z)b = (\Phi_{\vac,\be})_*(\cD(z)\vac\bt b)=b,$ proving \eqref{s2eq7}. Similarly,
\begin{align*}
 Y(a,z)\vac=(\Phi_{\al,\vac})_*(\cD(z)\bt\id_{X_\vac})(a\bt\vac)=\cD(z)(a)
\end{align*}
is holomorphic with $\cD(0)(a)=a$ for $z=0,$ proving \eqref{s2eq6}.
\vskip\topsep

\noindent
(b)\hskip\labelsep
We have already shown $\cD(z)\vac=\vac.$ To prove \eqref{s2eq8}, we first need a lemma.

\begin{lem}
For the universal complex line bundle\/ $\mathcal{L}\to\CP^\iy$ and\/ $n\in\Z$
\e
\label{s5eq1}
 \sum\nolimits_{k\geqslant0} t_k\cap i_{z,c_1^E(\mathcal{L})} F(z,c_1^E(\mathcal{L}))^n w^k =
 \sum\nolimits_{\ell\geqslant0} t_\ell\; i_{z,w}F(z,w)^n w^\ell.
\e
Moreover, for all $a\in E_*(X_\al),$ $b\in E_*(X_\be)$ we have
\ea
\label{s5eq2}
(\cD_\al(w)a\bt b)\cap C_z^E(\th_{\al,\be})&=(\cD_\al(w)\t \id_{X_\be})
 \bigl[(a\bt b)\cap i_{z,w}C_{F(z,w)}^E(\th_{\al,\be})\bigr],\\
\label{s5eq3}
(a\bt\cD_\be(w)b)\cap C_z^E(\th_{\al,\be})&=(\id_{X_\al}\t\,\cD_\be(w))\bigl[(a\bt b)\cap i_{z,w}C^E_{F(z,\io(w))}(\th_{\al,\be})\bigr].
\ea
\end{lem}

\begin{proof}
Introduce the expansion $i_{z,w}F(z,w)^n=\sum_{i\in\Z,j\geqslant0} F_{ij}^n\, z^iw^j.$ Then
\e
\label{s5eq4}
t_k\cap i_{z,c_1^E(\mathcal{L})}F(z,c_1^E(\mathcal{L}))^n=t_k\cap\sum\nolimits_{\substack{i\in\Z\\j\geqslant0}} F_{ij}^n\, z^i c_1^E(\mathcal{L})^j=\sum\nolimits_{\substack{i\in\Z\\j\geqslant0}} F_{ij}^n\, z^i t_{k-j},
\e
where $t_k=0$ for $k<0.$ Summing \eqref{s5eq4} over all $k,$ the summands with $k<j$ vanish, so we may restrict the sum to $k\geqslant j$ and reindexing by $\ell=k-j$ gives \eqref{s5eq1}:
\[
\sum\nolimits_{\substack{i\in\Z\\j\geqslant0}}\sum\nolimits_{\ell\geqslant 0} F_{ij}^n\, z^i w^j t_\ell w^\ell = \sum\nolimits_{\ell\geqslant 0} t_\ell\; i_{z,w}F(z,w)^n w^\ell
\]
For \eqref{s5eq2} we compute
\begin{align*}
 &(\cD_\al(w)a\bt b)\cap C_z^E(\th_{\al,\be})\overset{\mathclap{\eqref{s3eq3}}}{\;=\;}\sum\nolimits_{k\geqslant0} (\Psi_\al\t\id_{X_\be})_*(t_k\bt a\bt b)\cap C_z^E(\th_{\al,\be})w^k\\
 &\;=\;(\Psi_\al\t\id_{X_\be})_*\sum\nolimits_{k\geqslant0}(t_k\bt a\bt b)\cap (\Psi_\al\t\id_{X_\be})^*C_z^E(\th_{\al,\be})w^k\\ 
 &\overset{\mathclap{\eqref{s1eq9}}}{\;=\;}(\Psi_\al\t\id_{X_\be})_*\sum\nolimits_{k\geqslant0}(t_k\bt a\bt b)\cap C_z^E(\mathcal{L}\bt\th_{\al,\be})w^k\\
 &\overset{\mathclap{\eqref{s1eq6}}}{\;=\;}(\Psi_\al\t\id_{X_\be})_*\sum\nolimits_{k\geqslant0}(t_k\bt a\bt b)\cap i_{z,c_1^E(\mathcal{L})}C_{F(z,c_1^E(\mathcal{L}))}^E(\th_{\al,\be})w^k\\
  &\overset{\mathclap{\eqref{s5eq1}}}{\;=\;}(\Psi_\al\t\id_{X_\be})_*\sum\nolimits_{\ell\geqslant0}(t_\ell\bt a\bt b)\cap i_{z,w}C_{F(z,w)}^E(\th_{\al,\be})w^\ell\\
  &\;=\;(\cD_\al(w)\t \id_{X_\be})
 \bigl[(a\bt b)\cap i_{z,w}C_{F(z,w)}^E(\th_{\al,\be})\bigr].
\end{align*}
For \eqref{s5eq3} we similarly use \eqref{s1eq10} which replaces $c_1^E(\mathcal{L})$ by its formal inverse $\io(c_1^E(\mathcal{L}))$ above, so the same argument with $F(z,\io(w))$ in place of $F(z,w)$ gives \eqref{s5eq3}.
\end{proof}

\noindent
It is now easy to verify \eqref{s2eq8}: Let $a\in E_*(X_\al),$ $b\in E_*(X_\be).$ Then
\begin{align*}
 &Y(\cD_\al(w)a,z)b\overset{\mathclap{\eqref{s1eq14}}}{\;=\;}(\Phi_{\al,\be})_*(\cD_\al(z)\bt\id_\be)[(\cD_\al(w)a\bt b)\cap C_z^E(\th_{\al,\be})]\\
 &\overset{\mathclap{\eqref{s5eq2}}}{\;=\;}(\Phi_{\al,\be})_*(\cD_\al(z)\cD_\al(w)\bt\id_\be)
 \bigl[(a\bt b)\cap i_{z,w}C_{F(z,w)}^E(\th_{\al,\be})\bigr]\\
 &\overset{\mathclap{\eqref{s2eq4}}}{\;=\;} i_{z,w}Y(a,F(z,w))b.
\end{align*}
\vskip\topsep

\noindent
(c)\hskip\labelsep{}
Firstly, ${\Phi\circ(\Psi\t\Psi)\circ\de}\simeq{\Psi\circ(\Phi\t\id_{B\U(1)})}$ and $\De_*(t_k)=\sum_{i+j=k}t_i\bt t_j$ imply
\e
 \label{s5eq5}
 \cD_{\al+\be}(z)(\Phi_{\al,\be})_*=(\Phi_{\al,\be})_*(\cD_\al(z)\bt\cD_\be(z)).
\e
Let $a\in E_*(X_\al),$ $b\in E_*(X_\be),$ $c\in E_*(X_\ga).$ On the one hand
\begin{align*}
 &Y(Y(a,z)b,w)c=(\Phi_{\al+\be,\ga})_*(\cD_{\al+\be}(w)\bt\id_\ga)\\
 &\qquad\enskip\bigl[(\Phi_{\al,\be})_*(\cD_\al(z)\bt\id_\be)[(a\bt b)\cap C_z^E(\th_{\al,\be})]\bt c \cap C_w^E(\th_{\al+\be,\ga})\bigr]\\
 &\overset{\smash{\mathclap{\eqref{s5eq5}}}}{\quad=\quad}(\Phi_{\al+\be,\ga})_*(\Phi_{\al,\be})_*(\cD_\al(w)\bt\cD_\be(w)\bt\id_\ga)\\
 &\qquad\enskip\bigl[(\cD_\al(z)\bt\id_\be\bt\id_\ga)\bigl((a\bt b\bt c)\cap C_z^E(\th_{\al,\be})\cap (\Phi_{\al,\be}\t\id_\ga)^*C_w^E(\th_{\al+\be,\ga})\bigr)\bigr]\\
&\overset{\smash{\mathclap{\eqref{s2eq4}, \eqref{s5eq2}}}}{\quad=\quad}(\Phi_{\al+\be,\ga})_*(\Phi_{\al,\be})_*(\cD_\al(w)\cD_\al(z)\bt\cD_\be(w)\bt\id_\ga)\\
 &\qquad\enskip\bigl[(a\bt b\bt c)\cap C_z^E(\th_{\al,\be})\cap i_{w,z}C_{F(w,z)}^E(\th_{\al,\ga})\cap C_w^E(\th_{\be,\ga})\bigr],
\end{align*}
and on the other hand
\begin{align*}
 &i_{z,w} Y(a, F(z,w))Y(b,w)c=i_{z,w}(\Phi_{\al,\be+\ga})_*(\cD_\al(F(z,w))\bt\id_{\be+\ga})\\
 &\qquad\enskip\bigl[ \bigl( a\bt(\Phi_{\be,\ga})_*(\cD_\be(w)\bt\id_\ga)\bigl[(b\bt c)\cap C_w^E(\th_{\be,\ga})\bigr] \bigr)\cap C_{F(z,w)}^E(\th_{\al,\be+\ga}) \bigr]\\
 &\overset{\smash{\mathclap{\eqref{s5eq5},\eqref{s1eq8}}}}{\quad=\quad}i_{z,w}(\Phi_{\al,\be+\ga})_*(\id_\al\bt\Phi_{\be,\ga})_*(\cD_\al(F(z,w))\bt\id_\be\bt\id_\ga)\\
  &\qquad\enskip\bigl[ (\id_\al\bt\,\cD_\be(w)\bt\id_\ga)\bigl[(a\bt b\bt c) \cap C_w^E(\th_{\be,\ga})\bigr]\cap C_{F(z,w)}^E(\th_{\al,\be})\cap C_{F(z,w)}^E(\th_{\al,\ga}) \bigr]\\
 &\overset{\smash{\mathclap{\eqref{s2eq4},\eqref{s5eq3}}}}{\quad=\quad}(\Phi_{\al,\be+\ga})_*(\id_\al\bt\Phi_{\be,\ga})_*(\cD_\al(w)\cD_\al(z)\bt\cD_\be(w)\bt\id_\ga)\\
  &\qquad\enskip\bigl[(a\bt b\bt c) \cap C_w^E(\th_{\be,\ga})\cap C_z^E(\th_{\al,\be})\cap i_{z,w}C_{F(z,w)}^E(\th_{\al,\ga}) \bigr].
\end{align*}
As $Y(Y(a,z)b,w)c$ and $Y(a, F(z,w)) Y(b,w)c$ are both expansions in negative powers of $F(z,w)$ of the same series in different variables, there exist some $N \gg 0$ with $F(z,w)^N Y(Y(a,z)b,w)c = F(z,w)^N Y(a, F(z,w))Y(b,w)c,$ see \eqref{s2eq3}.\vskip\topsep

The same calculations show that \eqref{s1eq15} is a nonlocal vertex $F$-algebra and that the state-to-field correspondence $\overline{Y}(a,z)b$ preserves the degree shifted by $2\chi(\al,\al).$ It remains to prove $(-1)^{ab}\overline{Y}(a,z)b = \cD_{\al+\be}(z)\overline{Y}(b,\io(z)a).$ Notice $\si^*(\overline{C}^E_z(\th_{\al,\be}))=\overline{C}^E_{\io(z)}(\th_{\be,\al})$ for the swap $\si\colon X_\be\t X_\al\to X_\al\t X_\be.$ Using $\Phi_{\be,\al}\simeq\Phi_{\al,\be}\circ\si$ we find
\begin{align*}
&\cD_{\al+\be}(z)\overline{Y}(b,\io(z))a=\cD_{\al+\be}(z)(\Phi_{\be,\al})_*(\cD_\be(\io(z))\bt\id_\al)\bigl[(b\bt a)\cap \overline{C}^E_{\io(z)}(\th_{\be,\al})\bigr]\\
 &\;=\;\cD_{\al+\be}(z)(\Phi_{\al,\be})_*(\id_\al\bt\cD_\be(\io(z)))\si_*\bigl[(b\bt a)\cap\si^*\overline{C}^E_{z}(\th_{\al,\be})\bigr]\\
  &\overset{\smash{\mathclap{\eqref{s5eq5}}}}{\;=\;}(\Phi_{\al,\be})_*(\cD_\al(z)\bt\id_\be)\bigl[\si_*(b\bt a)\cap\overline{C}^E_{z}(\th_{\al,\be})\bigr]\\
  &\;=\;(\Phi_{\al,\be})_*(\cD_\al(z)\bt\id_\be)\bigl[(-1)^{ab}(a\bt b)\cap\overline{C}^E_{z}(\th_{\al,\be})\bigr]=(-1)^{ab}\overline{Y}(a,z)b.\hspace{1.15cm}\qed
\end{align*}

\begin{rem}
For the additive formal group law $\mathbb{G}_\mathrm{a}$ and ordinary homology, this was shown by Joyce \cite[Thm.~3.14]{Joy}. When $\X$ is the derived category of a finite quiver or of certain smooth projective complex varieties, then taking $F(X,Y) = X+Y$ in (\ref{s1eq15}) gives a (super) lattice vertex algebra \cite[Thm.~5.7]{Gro} \cite[Thm.~5.19]{Joy}.
\end{rem}

\begin{rem}
\label{StiefelWhitneyRemark}
A similar construction applies to H-spaces $X$ with $B\O(1)$-actions, the classifying space for \emph{real} line bundles, and homology with $\Z_2$-coefficients. Since $H^*(B\O(1))=\Z_2\llbracket \xi\rrbracket$ there is a shift operator $\cD(u)\colon H_*(X;\Z_2)\to H_*(X;\Z_2)\llbracket u\rrbracket$ for $u$ a variable of degree $-1.$ One can then build, just as in Theorem~\ref{s1thm1}, an operator $(-) \cap W_u(\th)$ of degree $-\rk\th_{\al,\be},$ where $\th_{\al,\be}\in KO(X_\al\t X_\be),$ with normalization $a\cap W_u(L)=a\cap(u+w_1(L))$ for the first Stiefel--Whitney class of a real line bundle $L\to X.$ Then $Y(a,z)b =(\Phi_{\al,\be})_*(\cD_\al(u)\bt\id_\be)\bigl[(a\ot b)\cap W_u(\th_{\al,\be})\bigr]$ makes $V=H_*(X;\Z_2)$ into a vertex algebra over $\Z_2$. 
\end{rem}

\subsection*{Acknowledgements.} The authors thank Dominic Joyce for many discussions and suggestions. They also thank Mikhail Kapranov, Kobi Kremnitzer, Sven Meinhardt, and Konrad Voelkel for helpful conversations. 

\medskip

\noindent J.~Gross, The Mathematical Institute, Radcliffe
Observatory Quarter, Woodstock Road, Oxford, OX2 6GG, U.K.\\
Email: {\tt jacob.gross@maths.ox.ac.uk}\\

\noindent M.~Upmeier, Department of Mathematics, University of Aberdeen, Fraser Noble Building, Elphinstone Rd, Aberdeen, AB24 3UE, U.K.\\E-mail: {\tt markus.upmeier@abdn.ac.uk.}

\end{document}